\documentclass[a4paper,12pt]{article}
\usepackage{amsfonts,amssymb}
\usepackage{amsmath}
\usepackage{mathtools}
\usepackage{mathrsfs}
\usepackage{amsthm}
\usepackage{thmtools}
\usepackage{faktor}
\usepackage{float}
\usepackage{physics}
\usepackage{subfig}
\usepackage{url}
\usepackage{authblk}
\usepackage{enumerate}
\usepackage[text={16.1cm,24cm},centering]{geometry} 
\usepackage[hidelinks]{hyperref}
\usepackage{titlesec}
\usepackage{authblk}
\usepackage[english]{babel}
\usepackage[utf8]{inputenc}
\theoremstyle{plain}
\newtheorem{theorem}{Theorem}
\newtheorem{prop}{Proposition}[section]
\newtheorem{corollary}{Corollary}[theorem]
\newtheorem{lemma}[prop]{Lemma}

\newtheorem{fact}[prop]{Fact}

\theoremstyle{definition}

\DeclareMathOperator{\sign}{sign}

\DeclareMathOperator{\col}{color}
\DeclareMathOperator{\flux}{flux}
\DeclareMathOperator{\twist}{Tw}
\newcommand{\D}{\mathcal{D}}
\newcommand{\tl}{\mathbf{t}}
\newcommand{\R}{\mathcal{R}}
\newcommand{\Pl}{\mathcal{P}}
\newcommand{\pl}{\mathbf{p}}

  \newcommand\extrafootertext[1]{%
    \bgroup
    \renewcommand\thefootnote{\fnsymbol{footnote}}%
    \renewcommand\thempfootnote{\fnsymbol{mpfootnote}}%
    \footnotetext[0]{#1}%
    \egroup
}
\begin{document}
\title{Domino tilings of three-dimensional cylinders: regularity of hamiltonian disks}

\author{Raphael de Marreiros}

\date{\today}

\affil{\small Departamento de Matemática, Pontifícia Universidade Católica do Rio de Janeiro\\
Rua Marquês de São Vicente, 225, Gávea, Rio de Janeiro, RJ 22451-900, Brazil \\
\url{raphaeldemarreiros@mat.puc-rio.br}}

\maketitle

\begin{abstract}
We consider three-dimensional domino tilings of cylinders $\D \times [0,N] \subset \mathbb{R}^3$, where $\D \subset \mathbb{R}^2$ is a balanced quadriculated disk and $N \in \mathbb{N}$.
A flip is a local move in the space of tilings: two adjacent and parallel dominoes are removed and then placed in a different position.
The twist is a flip invariant that associates an integer number to a domino tiling.
A disk $\D$ is called \emph{regular} if any two tilings of $\D \times [0,N]$ sharing the same twist can be connected through a sequence of flips once extra vertical space is added to the cylinder.
We prove that hamiltonian disks with narrow and small bottlenecks are regular.
In particular, we show that the absence of a bottleneck in a hamiltonian disk implies regularity.
\end{abstract}
\extrafootertext{2020 {\em Mathematics Subject Classification}.
Primary 05B45; Secondary 52C22, 05C70.}
\section{Introduction}
Problems involving domino tilings have been widely studied.
In particular, lots of results are known for two-dimensional domino tilings.
In this context, a region is a finite union of unit squares with vertices in $\mathbb{Z}^2$ and a \emph{domino} is a rectangle with sides of length one and two, formed by the union of two adjacent closed unit squares.
Kasteleyn~\cite{Kas61} established a connection between the number of \emph{domino tilings} (i.e., coverings by dominoes with disjoint interiors) of a region and the Pffafian of a skew-symmetric matrix, as a consequence, the exact number of domino tilings of a $m \times n$ rectangle is derived.
Independently, Temperley and Fisher~\cite{TF61} reached the same conclusion
using a different method.
Conway and Lagarias~\cite{CL90} and Thurston~\cite{THR90} used arguments from combinatorial group theory to develop a criterion to decide whether a region admits a tiling.
Cohn, Kenyon and Propp~\cite{CKP01} studied random tilings to investigate properties of a typical domino tiling of a large region.

We are particularly interested in the problem of connectivity of domino tilings via local moves.
A \emph{flip} is a local move involving two dominoes: two adjacent and parallel dominoes are removed and placed in a different position after a rotation of $90^{\circ}$.
The flip connectivity problem consists of characterizing the connected components under flips of the space of domino tilings $\mathcal{T}(\R)$ of a given region $\R$.
For two tilings $\tl_1,\tl_2 \in \mathcal{T}(\R)$ that can be connected by flips, we write $\tl_1\approx\tl_2$.

Thurston~\cite{THR90} proved that any two tilings of a \emph{quadriculated disk} (i.e., a planar region homeomorphic to a closed disk) can be joined by a sequence of flips.
Saldanha, Tomei, Casarin and Romualdo~\cite{STRD95} showed, for planar non-simply-connected regions, the existence of a flip invariant called flux.
Moreover, it is also proved that two domino tilings can be joined by a sequence of flips if and only if they have the same flux.

Domino tilings, along with the questions previously discussed, can be easily generalized to higher dimensions. 
However, the arguments used in two dimensions do not straightforwardly extend to higher dimensions.
Recently, Chandgotia, Sheffield and Wolfram~\cite{CSW23} developed new tools to extend the results in~\cite{CKP01} to 3D domino tilings.
In a different vein, the transition of dimensions can drastically change the a priori expected result.
For instance, Pak and Yang~\cite{PY13} showed that the counting tiling problem is computationally more complex in dimension three.
Similarly, as we shall see, the space of tilings of a contractible 3D region is not necessarily flip connected.

In the last decade, significant progresses have been made on the flip connectivity problem, particularly in regions such as cylinders~\cite{FKMS22, KS22, MS18, Sal22}.
A three-dimensional \emph{cylinder} $\R_N \subset \mathbb{R}^3$ is a region of the form $\D \times [0,N]$ where $\D \subset \mathbb{R}^2$ is a quadriculated disk and $[0,N]$ is an interval with $N \in \mathbb{N}$.
As illustrated in Figure~\ref{fig:drawtil}, we adhere to~\cite{MS18} and draw a tiling of $\R_N$ floor by floor.
The dominoes parallel to the $z$- axis, referred to as \emph{vertical dominoes}, correspond to two unit squares contained in adjacent floors.
The dominoes parallel to either the $x$-axis or the $y$-axis, referred to as \emph{horizontal dominoes}, are represented as planar dominoes.

\begin{figure}[ht]
\centerline{
\includegraphics[width=0.5\textwidth]{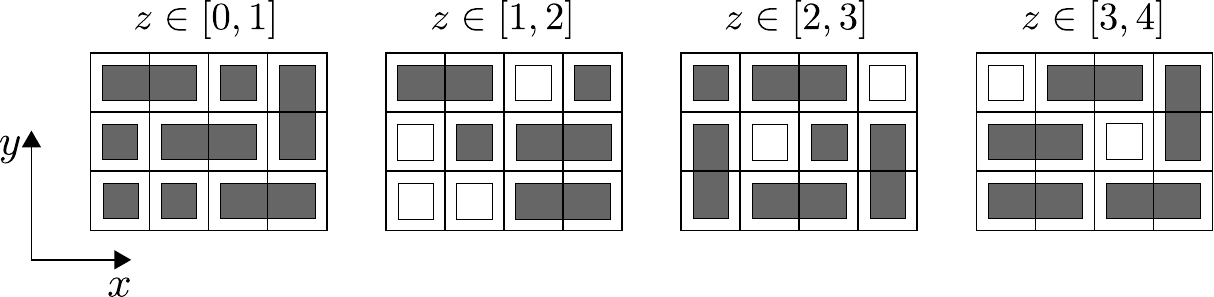}}
\caption{A domino tiling of the cylinder $[0,4]\times [0,3] \times [0,4]$.}
\label{fig:drawtil}
\end{figure}

Normally, we consider cylinders whose base is a balanced nontrivial disk.
A disk $\D$ is \emph{balanced} if it contains an equal number of black and white unit squares; a unit square $[a,a+1]\times[b,b+1] \subset \D$ with $(a,b) \in \mathbb{Z}^2$ is \emph{white} if $a+b$ is even and \emph{black} if $a+b$ is odd.
Moreover, $\D$ is \emph{trivial} if each of its unit squares is adjacent to at most other two unit squares.
It is not difficult to show that any two tilings of $\R_N$ can be joined by a sequence of flips if $\D$ is trivial.
For nontrivial disks, however, the discussion is more complex.

In order to study the flip connectivity of tilings of a cylinder, we consider an equivalence relation that is weaker than the relation $\approx$.
The definition of this relation requires a few concepts.
First, let the \emph{vertical tiling} $\tl_{\text{vert},2N} \in \mathcal{T}(\R_{2N})$ be the tiling formed exclusively by vertical dominoes.
Second, for tilings $\tl_1 \in \mathcal{T}(\R_{N_1})$ and $\tl_2 \in \mathcal{T}(\R_{N_2})$, let the \emph{concatenation} $\tl_1*\tl_2 \in \mathcal{T}(\R_{N_1+N_2})$ be the tiling formed by $\tl_1$ and the translation of $\tl_2$ by $(0,0,N_1)$.
We say that $\tl_1 \sim \tl_2$ if there are $M_1,M_2 \in 2\mathbb{N}$ such that $N_1+M_1=N_2+M_2$ and $\tl_1*\tl_{\text{vert},M_1} \approx \tl_2*\tl_{\text{vert},M_2}$.
Notice that if $\tl_1 \not\sim \tl_2$ then $\tl_1 \not\approx \tl_2$.
On the other hand, there exist tilings equivalent under $\sim$ but not under $\approx$; for an example, see Figure 1 of~\cite{FKMS22}.

The twist of a tiling, as introduced by Milet and Saldanha \cite{MS18}, is an invariant under flips defined for a large class of contractible regions contained in $\mathbb{R}^3$, including cylinders.
In this context, the twist assigns an integer number to each tiling.
The twist is closely related to the Hopf number, an invariant studied in physics whose existence is linked to the nontriviality of the third homotopy group of the sphere $\pi_3(S^2)=\mathbb{Z}$; for details, see Freedman, Hastings, Nayak and Qi~\cite{FHN11} and Bednik~\cite{Bed19}.
A more general definition of the twist, applicable to a broader class of regions, is presented by Freire, Klivans, Saldanha and Milet \cite{FKMS22}, using homology theory.

A balanced nontrivial disk $\D$ is \emph{regular} if any two tilings of $\R_N$ with the same twist can be joined by a sequence of flips once additional vertical space is added; equivalently, any two tilings with the same twist are equivalent under the relation $\sim$.
Saldanha~\cite{Sal22} showed that rectangles $[0,L] \times [0,M]$ with $LM$ even are regular if and only if $\min\{L,M\}>2$.
Furthermore, it is proved in~\cite{Sal21}, for a regular disk $\D$, that the cardinality of the largest connected component under flips of $\R_N$ is $\Theta(N^{-\frac{1}{2}}|\mathcal{T}(\R_N)|)$.

The domino group $G_\D$ and the even domino group $G_{\D}^+$ of a disk $\D$, first defined in~\cite{Sal22}, are fundamental to determine the regularity of $\D$.
As a set, $G_\D$ is defined as the quotient of $\bigcup_{N \geq 1} \mathcal{T}(\R_N)$ by the equivalence relation $\sim$.
The group operation in $G_\D$ is given by the concatenation, the identity element is $\tl_{\text{vert},2}$ and the inverse of a tiling $\tl$ is the tiling $\tl^{-1}$ obtained by reflecting $\tl$ on the $xy$-plane.
The even domino group $G_{\D}^+$ is the subgroup of $G_\D$ consisting of tilings with even height.
Consequently, $G_\D^+$ is a normal subgroup of index two of $G_{\D}$.

Notably, the twist defines a homomorphism $\textsc{TW} \colon G_\D \to \mathbb{Z}$, which maps $G_\D^+$ onto $\mathbb{Z}$ for nontrivial disks, ensuring that $G_\D$ is infinite.
It turns out that a disk $\D$ is regular if and only if the restriction of this homomorphism to $G_\D^+$ is an isomorphism.
In such cases, $G_\D$ is isomorphic to $\mathbb{Z} \oplus \mathbb{Z}/(2)$.

In this paper, we prove that a large class of disks is regular; in particular, the disks in Figures~\ref{fig:exreg} and~\ref{fig:regthm2}.
We show that two properties, both with graph-theoretical interpretations, determine the regularity of a disk.
Notice that there exists a natural identification between a disk $\D$ and a bipartite graph $\mathcal{G}(\D)$.
The vertices of $\mathcal{G}(\D)$ correspond to the unit squares in $\D$, and two vertices are connected by an edge if and only if their corresponding unit squares are adjacent (i.e., share an edge).
For simplicity, we refer to the properties of $\mathcal{G}(\D)$ as properties of $\D$.
In particular, $\D$ is called \emph{hamiltonian} if $\mathcal{G}(\D)$ has a hamiltonian cycle.

We also examine the presence of bottlenecks in $\D$, which are related to vertex cuts of size two of $\mathcal{G}_\D$.
A bottleneck is defined by a domino $d \subset \D$ that \emph{disconnects} $\D$, meaning that $\D \smallsetminus d$ is not connected.
It follows from~\cite{Mar23} that, in many cases, for disks with bottlenecks we have a surjective homomorphism from $G_\D^+$ to the free group of rank two $F_2$.
In such cases, there exists a constant $c \in (0,1)$ such that the cardinality of the largest flip connected component of $\R_N$ is $O(c^N|\mathcal{T}(\R_N)|)$.

Our first result establishes the regularity of bottleneck-free hamiltonian disks.
The second result shows the regularity of a hamiltonian disk obtained by introducing narrow and small bottlenecks into an initially bottleneck-free disk.
The strategy of the proofs is to consider a specific family of tilings that generates the even domino group and show that this family reduces to a single element.

\begin{theorem}\label{thm:regdisks}
Let $\D$ be a nontrivial hamiltonian quadriculated disk.
If for every domino $d \subset \D$ the region $\D \smallsetminus d$ is connected then $\D$ is regular.
\end{theorem}

\begin{figure}[H]
\centerline{
\includegraphics[width=0.65\textwidth]{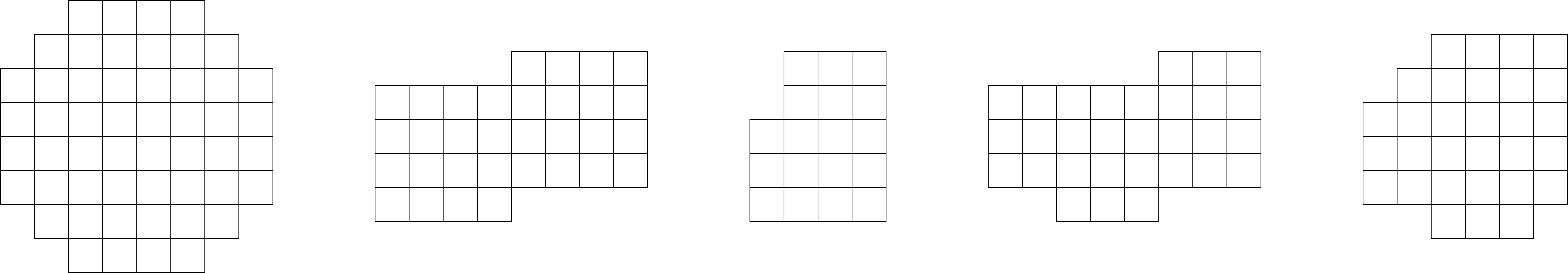}}
\caption{Examples of disks whose regularity follows from Theorem~\ref{thm:regdisks}.}
\label{fig:exreg}
\end{figure}

\begin{theorem}\label{thm:bottleneck}
Let $\D_0$ be a disk satisfying the hypothesis of Theorem~\ref{thm:regdisks}.
Consider pairwise disjoint disks $\D_1,\ldots, \D_k$ such that $|\D_i| < |\D_0| - 2$ and $\D_i \cap \D_0$ is a line segment of length two.
If $\D = \bigcup_{i=0}^k \D_i$ is a hamiltonian disk then $\D$ is regular.
\end{theorem}

\begin{figure}[H]
\centerline{
\includegraphics[width=0.65\textwidth]{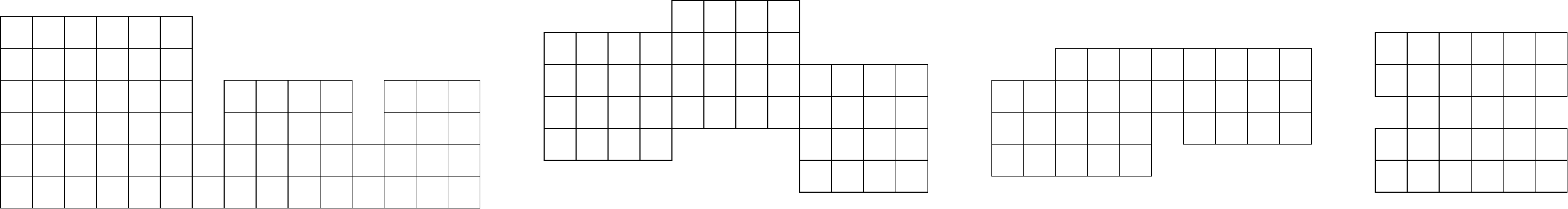}}
\caption{Examples of disks whose regularity follows from Theorem~\ref{thm:bottleneck}.}
\label{fig:regthm2}
\end{figure}

Naturally, Theorem~\ref{thm:regdisks} motivates the question of whether a bottleneck-free disk is hamiltonian, which is equivalent to determining whether a solid grid graph with no vertex cut of size two formed by adjacent vertices is hamiltonian.
Zamfirescu and Zamfirescu~\cite{ZZ92} showed that certain grid graphs with width greater than two are hamiltonian.
Conversely, Keshavarz-Kohjerdi and Bagheri~\cite{KKB23} established conditions that a hamiltonian rectangular truncated grid graph must satisfy.
As a consequence, the first two disks in Figure~\ref{fig:cexreg} are not hamiltonian.
On the other hand, it is easy to obtain examples of disks with small bottlenecks that fail to be hamiltonian, as the last two disks in Figure~\ref{fig:cexreg}.
Although our theorems do not apply directly, the regularity of these four disks can be easily established by combining our results with an additional analysis, as each disk becomes hamiltonian after the removal of certain two unit squares.
On a positive note, the inequality in Theorem~\ref{thm:bottleneck} regarding the size of the pairwise disjoint disks is tight (see Theorem 4 of~\cite{Mar23}).

\begin{figure}[H]
\centerline{
\includegraphics[width=0.5\textwidth]{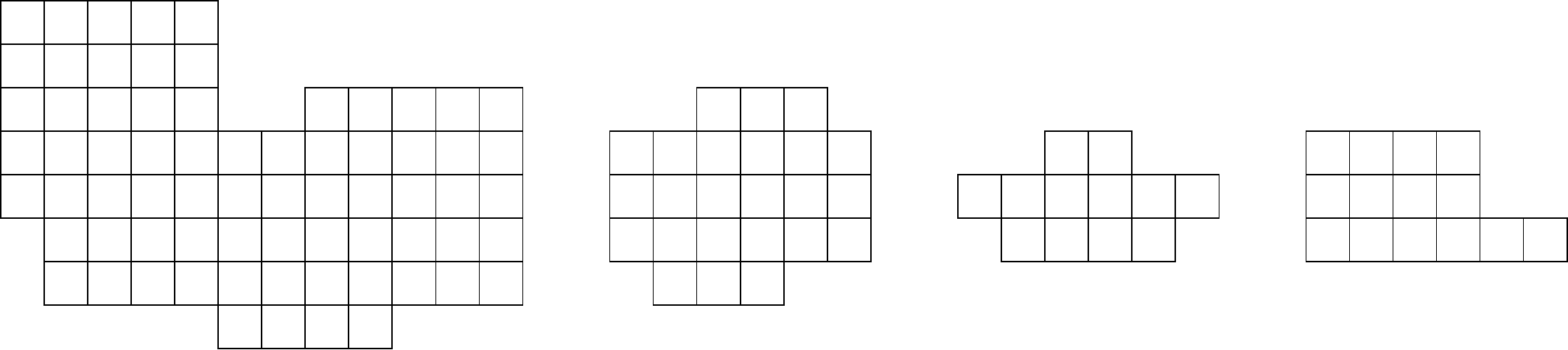}}
\caption{Four regular disks.}
\label{fig:cexreg}
\end{figure}

This paper is structured as follows.
In Section~\ref{sec:definitions} we present the construction of a family of generators for the domino group of a disk.
In Section~\ref{sec:hamiltoniandisks} we show some advantages of working with hamiltonian cycles and establish two technical lemmas.
We prove Theorems~\ref{thm:regdisks} and~\ref{thm:bottleneck} in Section~\ref{sec:thmproof}. 

The author thanks Nicolau Saldanha and Caroline Klivans for providing insightful comments and suggestions.
The author thanks for the support and hospitality of Brown University, where part of this work was completed.
The support of CNPq, FAPERJ, Projeto Arquimedes (PUC-Rio) and CAPES are appreciated.

\section{Definitions}\label{sec:definitions}

In this section, we provide the necessary background to establish Theorems~\ref{thm:regdisks} and~\ref{thm:bottleneck}.
The presented definitions are included to make the text self-contained, a more detailed explanation is available in~\cite{Sal22}.

Throughout this text we routinely abuse notation and neglect boundaries when discussing quadriculated regions.
Given two regions $\R$ and $\widetilde{\R}$, we write $\R \smallsetminus \widetilde{\R}$ for the quadriculated region formed by the closed unit squares that are in $\R$ but not in $\widetilde{\R}$.
Similarly, we say that $\R$ and $\widetilde{\R}$ are disjoint if they do not share any unit square.
In that sense, two unit squares that have only one vertex in common are said to be disjoint.

Henceforth, unless stated otherwise, we assume that all quadriculated disks are nontrivial and balanced.
Let $\D$ be a disk.
A \emph{plug} $p \subset \D$ is balanced subregion of $\D$, i.e., a union of an equal number of white and black closed unit squares in $\D$.
In particular, a domino is a plug.
We highlight the empty plug $\pl_{\circ}=\emptyset$ and the full plug $\pl_{\bullet}=\D$.
Given a plug $p$ we obtain another plug $p^{-1} = \D \smallsetminus p$.
We denote by $|p|$ the number of unit squares in $p$ and by $\Pl$ the set of plugs in $\D$.
A plug is \emph{compatible} with a domino $d \subset \D$ if they are disjoint.
The set of plugs compatible with $d$ is denoted by $\Pl_d$.

Consider $N \geq 2$ and two plugs $p_0, p_N \in \mathcal{P}$.
The \emph{cork} $\R_{0,N;p_0,p_N}$ is defined as:
$$\R_{0,N;p_0,p_N} = (\D \times [1,N-1]) \cup (p_0^{-1} \times [0,1]) \cup (p_N^{-1} \times [N-1,N]).$$
For instance, $\R_{0,N; \pl_{\circ}, \pl_{\circ}} = \R_N$.
A \emph{floor} is a triple $f=(p_1, f^*, p_2)$, where $p_1$ and $p_2$ are two disjoint plugs $f^*$ is a set of planar dominoes that defines a tiling of $\D \smallsetminus (p_1 \cup p_2)$.
Notice that the inverse $f^{-1} = (p_2, f^*, p_1)$ of a floor $f$ is also a floor.
Moreover, a floor $f$ is called \emph{vertical} if $f^* = \emptyset$.

There is an identification between tilings of corks (in particular, cylinders) and sequences of floors.
Indeed, as in Figure~\ref{fig:drawtil}, tilings are essentially drawn by exhibiting their corresponding sequences of floors.
Therefore, a tiling $\tl \in \mathcal{T}(\R_{0,N;p_1,p_2})$ can be described as a concatenation of floors: $\tl = f_1 * f_2* \ldots * f_N$.
Similarly, a tiling can be described as a concatenation of tilings of corks and floors.

An important fact is that pairs of vertical floors can be moved through flips.
For instance, consider a tiling $\tl = \tl_1*\tl_2$ with $\tl_1 \in \mathcal{T}(\R_{0,N_1; \pl_{\circ},p_1})$ and $\tl_2 \in \mathcal{T}(\R_{0,N_2; p_1, \pl_{\circ}})$.
It turns out that $\tl * \tl_{\text{vert},2} \approx \tl_1*\tl_{\text{vert},p_1} * \tl_2$, where $\tl_{\text{vert},p_1}$ is the tiling of $\R_{0,2;p_1,p_1}$ formed only by vertical dominoes (see Lemma 5.2 of~\cite{Sal22}).
Therefore, the relation $\sim$ allows the addition of an arbitrary even number of vertical floors between two floors of a tiling.

Consider a tiling $\tl \in \mathcal{T}(\R_{0,N; p_1, p_2})$.
The inverse of $\tl$ is the tiling $\tl^{-1} \in \mathcal{T}(\R_{0,N; p_2,p_1})$ obtained by reflecting $\tl$ on the $xy$ plane.
In the language of floors, if $\tl= f_1 * f_2 * \ldots * f_N$ then $\tl^{-1} = f_{N}^{-1} * f_{N-1}^{-1} * \ldots * f_1^{-1}$.
It follows from Lemma 4.2 of~\cite{Sal22} that $\tl * \tl^{-1} \sim \tl_{\text{vert}, p_1}$.

The rest of this section is dedicated to the study of the domino group.
First, we consider two large classes of disks.
Secondly, for a disk $\D$ in these classes, we construct a special family of tilings that generates the even domino group $G_{\D}^+$.

Let $\D$ be a disk.
A \emph{path} of length $k$ is characterized by a sequence $\gamma = (s_1, s_2, \ldots, s_k)$ of distinct unit squares in $\D$ such that $s_i$ and $s_{i+1}$ are adjacent for all $i$.
Additionally, $\gamma$ is called a \emph{cycle} if the initial and final squares $s_1$ and $s_{k}$ are also adjacent.
In that sense, a \emph{hamiltonian path} (resp.\ \emph{hamiltonian cycle}) of $\D$ (i.e., of the graph $\mathcal{G}(\D)$) is a path (resp.\ cycle) of length $|\D|$.
Figure~\ref{fig:hamiltoniandisks} shows examples of hamiltonian cycles and paths, their orientation and initial square are indicated by an~arrow.

\begin{figure}[H]
\centerline{
\includegraphics[width=0.55\textwidth]{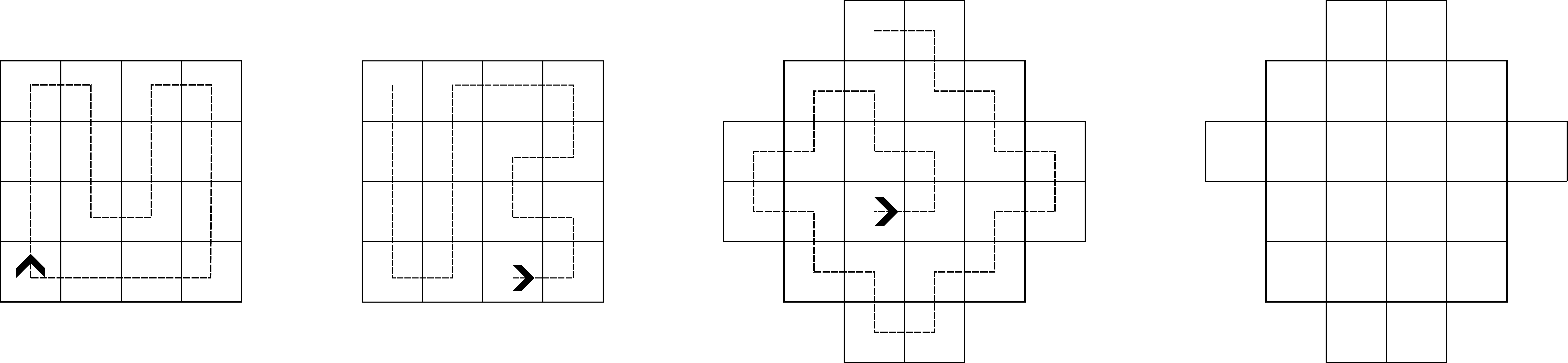}}
\caption{The first and the second example show a hamiltonian cycle and a hamiltonian path in $[0,4]^2$.
The third disk is path-hamiltonian but not hamiltonian.
The fourth disk is neither path-hamiltonian nor hamiltonian.}
\label{fig:hamiltoniandisks}
\end{figure}

From now on, until the end of this section, consider a fixed disk $\D$ with a hamiltonian path $\gamma = (s_1, s_2, \ldots, s_{|\D|})$; the following arguments apply similarly to hamiltonian cycles, as we can transform any cycle into a path by separating its initial and final squares.
We proceed towards the construction of a family of tilings that generates the even domino group $G_\D^+$.
We first define the notion of whether a domino respects $\gamma$.
A domino $d \subset \D$ \emph{respects} $\gamma$ if, for some $i$, it is equal the union of $s_i$ and $s_{i+1}$.
A three-dimensional domino $d \subset \R_N$ \emph{respects} $\gamma$ if its projection on $\D$ is either a unit square or a planar domino that respects $\gamma$.
In view of Fact~\ref{fact:resptil} below, it will be important to consider the planar dominoes in $\D$ that do not respect $\gamma$; the set of such dominoes is denoted by $\D_{\gamma}$. 

\begin{fact}[Lemma 8.1 of~\cite{Sal22}]\label{fact:resptil}
 Let $\D$ be a disk with a hamiltonian path $\gamma=(s_1, \ldots,s_{|\D|})$.
 If $\tl \in \mathcal{T}(\R_{0,2N;p,p})$ is a tiling whose dominoes respect $\gamma$ then $\tl \sim \tl_{\text{vert},p}$.
\end{fact}

For each plug $p \in \mathcal{P}$, we construct a tiling $\tl_p \in \mathcal{T}(\R_{0,|p|;p,\pl_{\circ}})$ whose dominoes respect~$\gamma$; by convention, $\tl_{\pl_{\circ}} = \emptyset$.
To this end, it suffices to describe the horizontal dominoes contained in each floor of $\tl_p$.
We proceed by induction on $|p|$.
Consider two unit squares of opposite colors $s_i,s_j \subset p$ with $i<j$ and $j-i$ minimal.
If $j=i+1$, the first floor (of $\tl_p$) contains no horizontal dominoes.
Otherwise, the horizontal dominoes of the first floor are obtained by placing dominoes along the path $(s_{i+1},s_{i+2}, \ldots, s_{j-1})$, more precisely, $s_{i+1}\cup s_{i+2}$,  $s_{i+3}\cup s_{i+4}$,$\ldots$, $s_{j-2}\cup s_{j-1}$.
Similarly, the horizontal dominoes of the second floor are obtained by placing dominoes along the path $(s_i, s_{i+1}, \ldots, s_j)$.
The concatenation of these two floors with $\tl_{p \smallsetminus (s_i \cup s_j)}$ (that we already constructed by induction) defines $\tl_p$; see Figure~\ref{fig:tp} for an example.
Notice that different choices of unit squares at minimal distance result in distinct tilings.
However, it follows from Fact~\ref{fact:resptil} that these tilings differ by a sequence of flips.

\begin{figure}[H]
\centerline{
\includegraphics[width=0.75\textwidth]{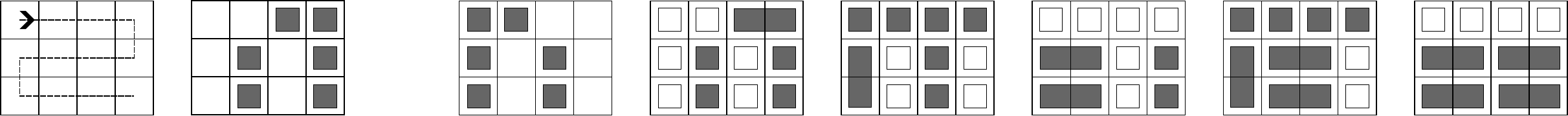}}
\caption{A disk with a hamiltonian path, a plug $p$ and the tiling $\tl_p$.}
\label{fig:tp}
\end{figure}

Consider a floor $f=(p_1,f^*,p_2)$.
Let $f_{\text{vert}}=(p_2, \emptyset, p_2^{-1})$ be a vertical floor.
Then, set $\tl_f= \tl_{p_1}^{-1} * f * f_{\text{vert}} * \tl_{p_2^{-1}} \in \mathcal{T}(\R_{N})$ where $N=|p_1|+|p_2^{-1}|+2$.
Notice that dominoes in $\tl_f$ that do not respect $\gamma$ must be in the floor $f$.

\begin{fact}[Lemma 8.2 of~\cite{Sal22}\protect\footnotemark]\label{fact: floors}
Let $\D$ be a disk with a hamiltonian path $\gamma$.
Consider $N$ even and a tiling $\tl \in \mathcal{T}(\R_{N})$ with floors $f_1, f_2, \ldots, f_N$, so that $\tl = f_1 * f_2* \ldots * f_{N}$.
Then, 
$$\tl \sim \tl_{f_1} * \tl_{f_2^{-1}}^{-1} * \ldots * \tl_{f_i^{(-1)^{(i+1)}}}^{(-1)^{(i+1)}} * \ldots *\tl_{f_N^{-1}}^{-1}.$$
\end{fact}
\footnotetext{There is a typo in the original result, and the $-1$ superscripts on the floors of even parity are missing.}

We now consider a particular case of the construction above.
Consider a domino $d  \subset \D$ with a compatible plug $p \in \mathcal{P}_d$.
Let $f=(p, d, (p\cup d)^{-1})$ be a floor and set $\tl_{d,p; \gamma}=\tl_f$; when the context is clear we write $\tl_{d,p}$ instead of $\tl_{d,p; \gamma}$.
If $d \not\in \D_{\gamma}$ then Fact~\ref{fact:resptil} implies that $\tl_{d,p} \sim \tl_{\text{vert}}$.
However, if $d \in \D_{\gamma}$ then $d \times [|p|, |p|+1]$ is the only domino in $\tl_{d,p}$ that does not respect $\gamma$; for instance, see Figure~\ref{fig:ex_generator}.

\begin{figure}[H]
\centerline{
\includegraphics[width=0.79\textwidth]{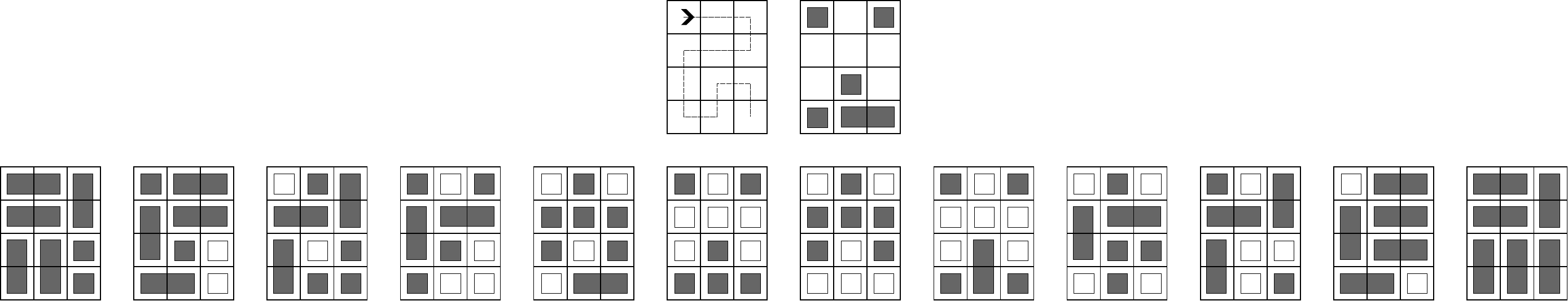}}
\caption{The first row shows the disk $\D=[0,3]\times [0,4]$ with a hamiltonian path $\gamma$, and a domino $d$ with a compatible plug $p$.
The second row shows $\tl_{d,p;\gamma} = \tl_p^{-1} * f * f_{\text{vert}} * \tl_{p \cup d}$.}
\label{fig:ex_generator}
\end{figure}

\begin{fact}[Lemma 8.3 of~\cite{Sal22}]\label{fact:gen}
Let $\D$ be a disk with a hamiltonian path $\gamma$.
Consider a floor $f=(p, f^*, \tilde{p})$.
Suppose that $f^* = \{ d_1, d_2, \ldots, d_k \}$.
Let $p_1 = p$ and $p_{i+1} = p_i \cup d_i$.
Then,
$$ \tl_f \sim \tl_{d_1,p_1} * \ldots * \tl_{d_i,p_i} * \ldots * \tl_{d_k,p_k}.$$
\end{fact}

Facts~\ref{fact: floors} and \ref{fact:gen} imply that the even domino group $G_{\D}^+$ is generated by the set of tilings of the form $\tl_{d,p}$.
It is possible to reduce this family of generators through the flux between a domino and a plug, which we now define.
In order to define the flux, it is important to distinguish whether a unit square is black or white.
To facilitate this distinction, we utilize the path $\gamma$.
For a unit square $s_i$, define $\col(s_i)=(-1)^i$; then, $s_i$ is identified as white if $\col(s_i)=+1$ and as black if $\col(s_i)=-1$.

Consider a domino $d \in \D_{\gamma}$, so that $d=s_k \cup s_l$ with $l-k>1$.
The domino $d$ decomposes the region $\D \smallsetminus d$ into three subregions:
$$\D_{d,-1} = \bigcup_{i=1}^{k-1} s_i, \quad \quad \D_{d,0} = \bigcup_{i=k+1}^{l-1} s_i, \quad \quad \D_{d,+1} = \bigcup_{i=l+1}^{|\D|} s_i.$$
Clearly, each region $\D_{d,j}$ comes with a hamiltonian path $\gamma_{d,j}$.
The regions $\D_{d,-1}$ and $\D_{d,+1}$ are not necessarily balanced and nonempty.
However, $\D_{d,0}$ is always balanced and nonempty, as $(s_k,s_{k+1}, \ldots, s_l)$ is a cycle.
The union of $\D_{d,-1}$ and $\D_{d,+1}$ is denoted by~$\D_{d, \pm 1}$.

For $p \in \Pl_d$, we define a triple $\flux(d,p)= (\flux_{-1}(d,p), \flux_{0}(d,p), \flux_{+1}(d,p)) \in \mathbb{Z}^3$.
The coordinate $\flux_j(d,p)$ is computed by summing $\col(s_i)=(-1)^i$ over the unit squares $s_i$ contained in $p \cap \D_{d,j}$.
Notice that $\flux_{-1}(d,p) + \flux_{0}(d,p) + \flux_{+1}(d,p)=0$, as $p$ is a balanced subregion of $\D$.
It turns out that if $p_1, p_2 \in \Pl_d$ and $\flux(d,p_1)= \flux(d,p_2)$ then $\tl_{d,p_1} \sim \tl_{d,p_2}$ (see Lemma 8.4 of~\cite{Sal22}).
We then have the following fact. 

\begin{fact}[Corollary 8.6 of~\cite{Sal22}]\label{fact:completefamilytil}
Consider a disk $\D$ with a hamiltonian path $\gamma$.
For each domino $d \in \D_{\gamma}$, let $\Phi_d$ be the set of the possible triples $\flux(d, \cdot)$.
For each $\phi \in \Phi_d$, consider a plug $p_{d, \phi} \in \mathcal{P}_d$ such that $\flux(d,p_{d,\phi}) = \phi$.
The even domino group $G_\D^+$ is generated by the family of tilings $(\tl_{d,p_{d, \phi}})$.
\end{fact}

\section{Hamiltonian disks}\label{sec:hamiltoniandisks}
In this section, we present the advantages of using hamiltonian disks and establish two technical lemmas.
The following lemma demonstrates that the relative position between two dominoes that do not respect a cycle must satisfy certain properties.

\begin{lemma}\label{lem:cyclelength} 
Let $\D$ be a disk with a hamiltonian cycle $\gamma=(s_1, s_2, \ldots, s_{|\D|})$.
Consider two dominoes $\tilde{d}=s_i \cup s_j$ and $d=s_k \cup s_l$ that do not respect $\gamma$.
Suppose that $l-k \geq 5$.
If $i < k < j < l$ (resp. $k < j < l < i$) then $\max\{ k-i,l-j \} \geq 3$ (resp. $\max\{ j-k,i-l \} \geq 3$).
\end{lemma}

\begin{proof}
Let $i<k<j<l$, the other case is similar.
Suppose, for a contradiction, that $k-i < 3$ and $l-j < 3$.
Since $|\D_{\tilde{d},0}|$ is even and positive, $j-i$ is an odd number; analogously, $l-k$ is odd.
Thus, either $k-i = 1 = l-j$ or $k-i=2=l-j$.
In the first case, $(s_i, s_k, s_l, s_j)$ defines a cycle of length four and therefore corresponds to a $2 \times 2$ square in $\D$.
In the second case, $(s_k,s_{i+1},s_i, s_j,s_{j+1},s_l)$ defines a cycle of length six and therefore corresponds to a $3\times 2$ rectangle in $\D$.
Then, $\gamma$ essentially follows one of the four patterns shown in Figure~\ref{fig:possiblepaths}.

\begin{figure}[H]
\centerline{
\includegraphics[width=0.45\textwidth]{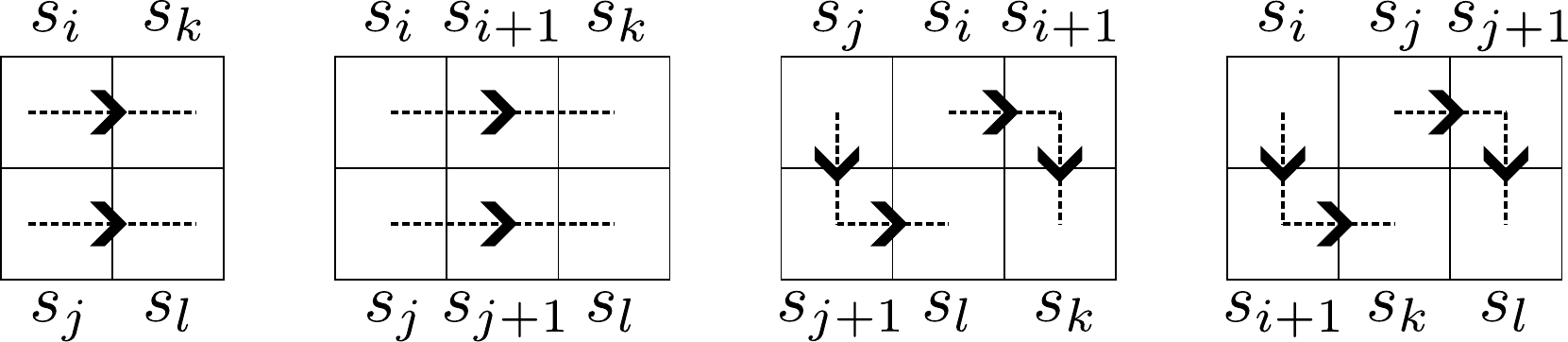}}
\caption{Four possible configurations of $\gamma$.}
\label{fig:possiblepaths}
\end{figure}

However, in any of these four possibilities we have a contradiction with the fact that $s_1$ and $s_{|\D|}$ are adjacent.
Indeed, in the first three cases the cycle $(s_i, \ldots, s_j)$ defines two regions in $\D$ that separate $s_1$ and $s_{|\D|}$.
The fourth case follows by the same argument once we observe that $s_j$ and $s_k$ are not adjacent, as $l-k \geq 5$.
\end{proof}

Let $\D$ be a disk with a hamiltonian cycle $\gamma = (s_1, s_2, \ldots, s_{|\D|})$.
By relabeling the unit squares, we can choose any unit square in $\D$ to be the initial square of $\gamma$.
We will need that a hamiltonian cycle starts at a corner, i.e., a unit square that is adjacent to only other two unit squares.
Let $s_{SW}$ be the southwesternmost unit square in $\D$, i.e.,
\begin{gather*}
s_{SW} = [a,a+1] \times [b,b+1] \text{, where} \\
b=\min \{ y \in \mathbb{Z} \colon (x,y) \in \D \text{ for some } x \in \mathbb{Z} \} \text{ and }a = \min \{ x \in \mathbb{Z} \colon (x,b) \in \D \}
\end{gather*}
Notice that $s_{SW}$ is a corner of $\D$.
Moreover, if $s$ is a unit square adjacent to $s_{SW}$ then $\D \smallsetminus (s\cup s_{SW})$ is a path-hamiltonian disk.

A hamiltonian cycle corresponds to a simple closed curve in $\mathbb{R}^2$ that passes through the center of every unit square in $\D$, as in Figure~\ref{fig:hamiltoniandisks}.
Therefore, $\gamma$ divides the plane into two connected components, one of the components is bounded and the other is unbounded.
The bounded connected component is called the \emph{interior} of $\gamma$, the unbounded component is called the \emph{exterior} of $\gamma$.
We say that a domino $d=s_k \cup s_l$ is contained in the interior (resp.\ exterior) of $\gamma$ if the line segment between the centers of $s_k$ and $s_l$ intersects the interior (resp.\ exterior) of $\gamma$; for instance, see Figure~\ref{fig:domint}.
Notice that, except by $s_1 \cup s_{|\D|}$, every domino that does not respect $\gamma$ is contained in either the interior or the exterior~of~$\gamma$.
This fact help us to prove the following result.

\begin{figure}[H]
\centerline{
\includegraphics[width=0.38\textwidth]{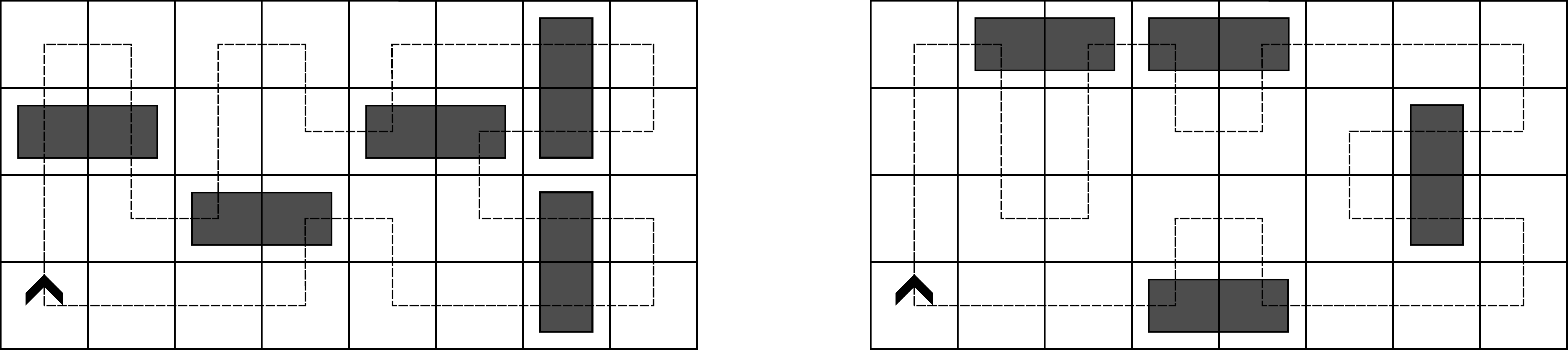}}
\caption{Examples of dominoes in the interior and exterior of a cycle, respectively.}
\label{fig:domint}
\end{figure}

\begin{lemma}\label{lem:alternatingdominoes}
Let $\D$ be a disk with a hamiltonian cycle $\gamma=(s_1,\ldots,s_{|\D|})$ where $s_1=s_{SW}$.
Consider two distinct dominoes $d, \tilde{d} \in \D_{\gamma}$.
There exist dominoes $d_1,d_2,\ldots,d_n$ with $d_1=\tilde{d}$ and $d_n=d$ such that, for each $i \in \{1,2,\ldots,n-1\}$, one of the following~holds:
\begin{enumerate}
    \item $d_{i+1}$ is the union of a unit square in $\D_{d_i,0}$ and a unit square in $\D_{d_i, \pm 1}$.
    
    \item $d_{i+1} \subset \D_{d_i,0}$ and $\D_{d_i,0} \smallsetminus \D_{d_{i+1},0}$ is a disk.

    \item  $d_{i+1} \subset \D_{d_i,\pm 1}$ and $\D_{d_i,\pm 1} \smallsetminus \D_{d_{i+1},\pm 1}$ is a disk.
\end{enumerate}
\end{lemma}

\begin{proof}   
Throughout the proof, we say that two dominoes $\bar{d}_1, \bar{d}_2 \in \D_{\gamma}$ form a good pair if the result holds for $d=\bar{d}_1$ and $\tilde{d}=\bar{d}_2$.
We need to show that any two dominoes that do not respect $\gamma$ form a good pair.
The proof is by induction on $|\D|$, with the base case being when $|\D|=4$.
In this case, $\D$ is a $2\times 2$ square and the result follows vacuously.

In order to prove the induction step, we first consider the case that there exists a domino $\hat{d}$ that disconnects $\D$; clearly, $\hat{d} \in \D_{\gamma}$.
By possibly changing the orientation of~$\gamma$ ($s_1=s_{SW}$ still holds) we may assume that $\hat{d}$ and $s_1 \cup s_{|\D|}$ are disjoint.
Notice that $\hat{d}$ defines two other disks $\D_1 = \D \smallsetminus \D_{\hat{d},0}$ and $\D_2 = \D  \smallsetminus \D_{\hat{d},\pm 1}$ such that $\D=\D_1 \cup \D_2$ and $\D_1 \cap \D_2 = \hat{d}$; for instance, see Figure~\ref{fig:altdom0}. 
The cycle $\gamma$ induces hamiltonian cycles $\gamma_1$ in $\D_1$ and $\gamma_2$ in $\D_2$; the union of the initial and final squares of $\gamma_1$ and $\gamma_2$ constitute $s_1\cup s_{|\D|}$ and $\hat{d}$, respectively. 
Hence, by the induction hypothesis, it suffices to show that $s_1\cup s_{|\D|}$ and $\hat{d}$ form a good pair.
On the other hand, since $\hat{d}$ is disjoint from $s_1\cup s_{|\D|}$, we have that $\D_{s_1\cup s_{|\D|},0} \smallsetminus \D_{\hat{d},0} = \D_1 \smallsetminus (s_1 \cup s_{|\D|})$ is a disk.

\begin{figure}[H]
\centerline{
\includegraphics[width=0.5\textwidth]{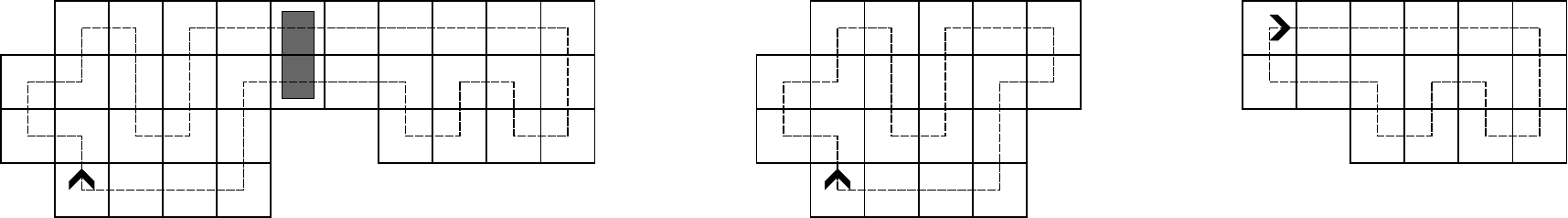}}
\caption{A hamiltonian disk $\D$ with a domino $\hat{d}$ that disconnects $\D$ and the induced hamiltonian disks $\D \smallsetminus \D_{\hat{d},0}$ and $\D \smallsetminus \D_{\hat{d},\pm 1}$.}
\label{fig:altdom0}
\end{figure}

Now, suppose there exists no domino that disconnects $\D$.
Therefore, every domino $\hat{d} \in \D_{\gamma} \smallsetminus \{s_1 \cup s_{|\D|}\}$ forms a good pair with a domino $\bar{d}$ composed of a unit square in $\D_{\hat{d},0}$ and a unit square in $\D_{\hat{d}, \pm 1}$.
Notice that if $\hat{d}$ is contained in the interior of $\gamma$ then $\bar{d}$ is contained in the exterior of $\gamma$, and vice versa.
By possibly altering the orientation of $\gamma$, as in the previous paragraph, it is not difficult to see that the $2\times 2$ square that contains $s_1=s_{SW}$ also contains a domino that forms a good pair with $s_1 \cup s_{|\D|}$.
Thus, the result follows once we prove that any two dominoes in the exterior of $\gamma$ form a good pair.

Let $\partial \D \subset \D$ be the union of the unit squares not contained in the topological interior of~$\D$.
Let $d^1,d^2, \ldots, d^j$ be the dominoes in $\partial \D$ that are distinct from $s_1 \cup s_{|\D|}$ and do not respect~$\gamma$; see Figure~\ref{fig:altdom}.
Clearly, these dominoes are contained in the exterior of $\gamma$.
Notice that every domino in the exterior of $\gamma$ is contained in the pairwise disjoint union $\bigcup\limits_{m=1}^j\D_{d^m,0} \cup d^m$.
Since $\gamma$ induces a hamiltonian cycle in $\D_{d^m,0} \cup d^m$, it follows from the the induction hypothesis that $d^m$ form a good pair with every domino in $\D_{d^m,0}$ that does not respect $\gamma$.
We are left to show that $d^r$ and $d^s$ form a good pair for any $r,s \in \{1,2,\ldots, j\}$.

\begin{figure}[H]
\centerline{
\includegraphics[width=0.45\textwidth]{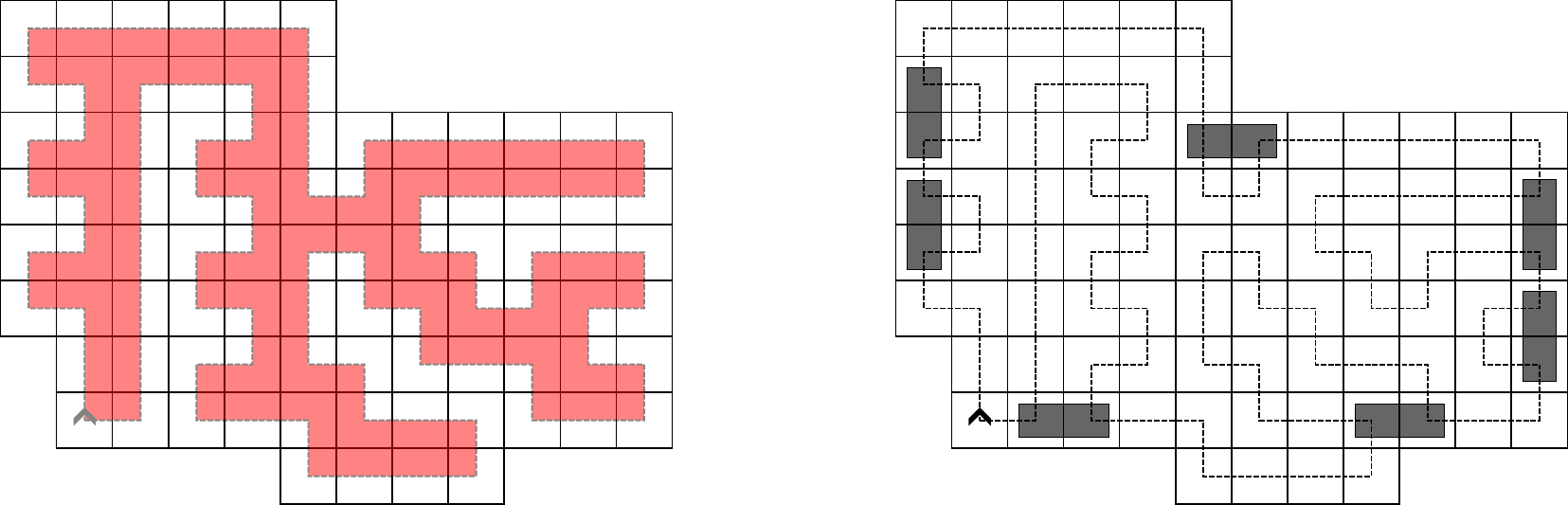}}
\caption{A disk $\D$ with the interior of a hamiltonian cycle shown in red, and the dominoes in $\partial \D$ that are distinct from $s_1 \cup s_{|\D|}$ and do not respect $\gamma$.}
\label{fig:altdom}
\end{figure}

The fact that there is no domino that disconnects $\D$ implies that every non-corner unit square in $\partial \D$ is adjacent to a unit square not in $\partial \D$, so that the region $\D \smallsetminus \partial \D$ is connected.
Therefore, by possibly relabeling the dominoes $d^1,d^2,\ldots,d^j$, we may assume that for each $m$ there is a unit square in $\D_{d^{m+1},0}$ adjacent to a unit square in $\D_{d^{m},0} \subset \D_{d^{m+1}, \pm 1}$.
Hence, $d^m$ and $d^{m+1}$ form a good pair.
\end{proof}

\section{Regularity of disks}\label{sec:thmproof}
In this section, we prove Theorems~\ref{thm:regdisks} and~\ref{thm:bottleneck}.
We demonstrate that if a disk $\D$ with a hamiltonian cycle $\gamma$ satisfies the conditions of these theorems then the even domino group $G_{\D}^+$ is cyclic.
To this end, we reduce the family of generators of $G_{\D}^+$, defined in Section~\ref{sec:definitions}, to only one element.
Recall that this family is composed of tilings $\tl_{d,p}$ where $d$ is a domino that does not respect $\gamma$ and $p$ is a plug compatible with $d$.
The next two results provide an initial reduction of this family of generators.
Indeed, Lemmas~\ref{lem:tdvazio} and~\ref{lem:fluxinvariance} imply that,
instead of considering all possible triples $\flux(d, \cdot)$, it suffices to consider all possible nonzero integers $\flux_0(d, \cdot)$.

\begin{lemma}\label{lem:tdvazio}
Let $\D$ be a disk with a hamiltonian cycle $\gamma=(s_1,\ldots, s_{|\D|})$ where $s_1=s_{SW}$.
Consider a domino $d \in \D_{\gamma}$ and a plug $p \in \mathcal{P}_d$.
If $\flux_0(d,p)=0$ then~$\tl_{d, p} \sim \tl_{\text{vert}}$.
\end{lemma}

\begin{proof}
We may assume $p \cap \D_{d,0} = \emptyset$.
Otherwise, for $\tilde{p} = p\cap \D_{d,\pm 1}$ we have $\tilde{p} \cap \D_{d,0} =\emptyset$ and $\flux(d,p)=\flux(d,\tilde{p})$, consequently $\tl_{d,p}\sim \tl_{d,\tilde{p}}$.
Then, by construction, the tiling $\tl_{d,p}$ covers the region $\D_{d,0} \times [|p|,|p|+2]$ solely with vertical dominoes.
Let $\tl_1$ be the tiling obtained from $\tl_{d,p}$ by performing vertical flips along $\gamma_{d,0}$, so that the restriction of $\tl_1$ to $(\D_{d,0} \cup d) \times [|p|,|p|+1]$ is occupied only by horizontal dominoes.
Notice that $\D_{d,0} \cup d$ is a disk.
Indeed, $\D_{d,0} \cup d$ does not enclose a unit square, as $\D_{d,0} \cup d$ is defined by the cycle $(s_k, s_{k+1}, \ldots, s_l)$ and neither $s_1$ nor $s_{|\D|}$ is contained in the interior of $\D$.
Since the space of tilings of a disk is connected under flips, there exists a sequence of horizontal flips that takes $\tl_1$ to the tiling $\tl_2$ whose restriction to $(\D_{d,0} \cup d )\times [|p|,|p|+1]$ corresponds to the tiling induced by $(s_k, s_{k+1}, \ldots, s_l)$.
Thus, $\tl_{d,p} \approx \tl_1 \approx \tl_2$ and every domino in $\tl_2$ respects~$\gamma$.
The result now follows from Fact~\ref{fact:resptil}.
\end{proof}

\begin{lemma}\label{lem:fluxinvariance}
Let $\D$ be a disk with a hamiltonian cycle $\gamma=(s_1,\ldots, s_{|\D|})$ where $s_1= s_{SW}$.
Consider a domino $d \in \D_{\gamma}$ and $p_1,p_2 \in \mathcal{P}_d$.
If $\flux_0(d,p_1)=\flux_0(d,p_2)$ then $\tl_{d,p_1} \sim \tl_{d,p_2}$.
\end{lemma}

\begin{proof}
Throughout this proof, in addition to $\gamma$, we also work with another hamiltonian cycle $\tilde{\gamma}$.
To avoid confusion, we distinguish the flux between a domino and plug with respect to each cycle.
Specifically, we denote the flux with respect to $\tilde{\gamma}$ and $\gamma$ by $\flux(\cdot, \cdot ; \tilde{\gamma})$ and $\flux(\cdot, \cdot ;\gamma)$, respectively.
Similarly, we refer to the three subregions of $\D$ determined by a domino that does not respect a given cycle.

Suppose that $d=s_k \cup s_l$ with $k<l$; notice that $l-k \geq 3$.
Consider the hamiltonian cycle $\tilde{\gamma} = (s_{k+1}, s_k, \ldots, s_1, s_{|\D|}, s_{|\D|-1}, \ldots, s_{k+2})$ obtained by changing the initial square of~$\gamma$; for instance, see Figure~\ref{fig:pathalteration}.

\begin{figure}[H]
\centerline{
\includegraphics[width=0.35\textwidth]{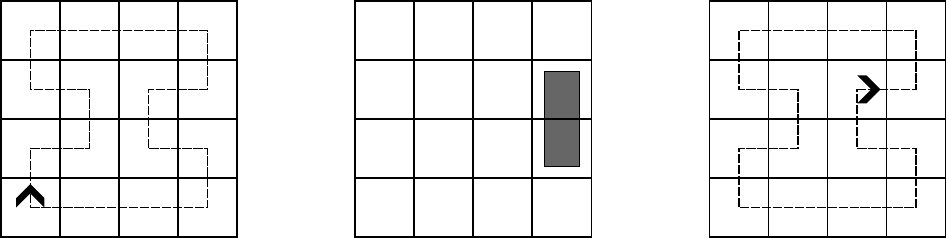}}
\caption{A disk with a hamiltonian cycle $\gamma$, a domino $d$ and the cycle $\tilde{\gamma}$ obtained from~$\gamma$.}
\label{fig:pathalteration}
\end{figure}

We claim that if $p$ is a plug compatible with $d$ then $\tl_{d,p;\gamma} \sim \tl_{d,p;\tilde{\gamma}}$.
Notice that a domino in $\tl_{d,p; \tilde{\gamma}}$ which does not respect $\gamma$ and does not project on $d$ is of the form $d_1 \times [N-1,N]$ with $d_1= s_1 \cup s_{|\D|}$.
Then, $\tl_{d,p; \tilde{\gamma}}$ is $\sim \text{-}$ equivalent to a concatenation of $\tl_{d,p;\gamma}$ and possibly tilings $\tl_{d_1,\hat{p}; \gamma}$.
However, for any plug $\hat{p}$ compatible with $d_1$, we have $\flux(d_1,\hat{p};\gamma)= (0,0,0)$.
Thus, by Lemma~\ref{lem:tdvazio}, $\tl_{d_1,\hat{p};\gamma} \sim  \tl_{\text{vert}}$.
Therefore, $\tl_{d,p;\gamma} \sim \tl_{d,p;\tilde{\gamma}}$.

Let $\hat{p_1}$ be the plug formed by the union of the unit squares in $p_2 \cap \D_{d,0;\gamma}$ and $p_1 \cap \D_{d, \pm 1;\gamma}$.
Notice that $\flux(d,\hat{p_1};\gamma)=\flux(d,p_1;\gamma)$ and $\flux(d,\hat{p_1};\tilde{\gamma})= \flux(d,p_2;\tilde{\gamma})$.
Therefore, we have $\tl_{d,p_1; \gamma} \sim \tl_{d,\hat{p_1}; \gamma} \sim \tl_{d,\hat{p_1}; \tilde{\gamma}} \sim \tl_{d,p_2; \tilde{\gamma}} \sim \tl_{d,p_2;\gamma}$.
\end{proof}

We now deal with tilings that have large flux, more specifically, tilings $\tl_{d,p}$ such that $|\flux_0(d,p)| \geq 2$.
In Lemma~\ref{lem:generatorsfluxgeq2}, we show that in some cases such tilings can be decomposed as a concatenation of tilings with smaller flux.
As a consequence, it follows that the even domino group of a bottleneck-free disk is generated by tilings with $|\flux_0(\cdot,\cdot)|= 1$.

\begin{lemma}\label{lem:generatorsfluxgeq2}
Let $\D$ be a disk with a hamiltonian cycle $\gamma=(s_1,\ldots, s_{|\D|})$ where $s_1=s_{SW}$. 
Consider a domino $d \in \D_{\gamma}$ such that $\D \smallsetminus d$ is connected.
Let $p$ be a plug compatible with $d$ such that $|\flux_0(d,p)| \geq 2$.
Then, there exist plugs $p_0, p_1, p_2 \in \mathcal{P}$ and a domino $\tilde{d} \in \D_{\gamma}$ that does not disconnect~$\D$ such that: 
\begin{enumerate}

    \item $\tl_{d,p} \sim \tl_{\tilde{d},p_1}^{\epsilon_1} * \tl_{d,p_0} * \tl_{\tilde{d},p_2}^{\epsilon_2}$ for some $\epsilon_1, \epsilon_2 \in  \{ -1, 1\}$.

    \item $\max\{ |\flux_0(\tilde{d},p_1)|,|\flux_0(d,p_0)|, |\flux_0(\tilde{d},p_2)| \} < |\flux_0(d,p)|$.

\end{enumerate}
\end{lemma}

\begin{proof}
Since $\D \smallsetminus d$ is connected, there exists a domino $\tilde{d}=s_i \cup s_j$
with $s_i \in \D_{d,\pm 1}$ and $s_j \in \D_{d,0}$.
We have $s_i \neq s_1$; otherwise the fact that $s_1$ is a corner implies that $s_j = s_{\D}$, which is not an element of $\D_{d,0}$.
Notice that $\D \smallsetminus \tilde{d}$ is also connected, as $d$ is formed by unit squares in $\D_{\tilde{d},0}$ and $\D_{\tilde{d},\pm 1}$.
Suppose that $s_i \in \D_{d,-1}$, the case $s_i \in \D_{d,+1}$ is analogous.

Consider first the scenario $\col(s_j) = \sign(\flux_0(d,p))$ (recall that $\col(s_j)=(-1)^j$).
Suppose $d=s_k \cup s_l$, with $k<l$.
Since $|\flux_0(d,p)| \geq 2$, we have $l-k \geq 5$.
By Lemma~\ref{lem:cyclelength}, it follows that $\max\{ k-i, l-j \} \geq 3$.
Consequently, by Lemma~\ref{lem:fluxinvariance}, we may assume, without loss of generality, that $p$ satisfies the following conditions:
\begin{enumerate}[(i)]
    \item $p$ contains $s_i \cup s_j$ and $p \cap \D_{d,0}$ (resp. $p \cap \D_{d, \pm 1}$) consists only of unit squares whose color matches $\col(s_j)$ (resp. $\col(s_i)$).

    \item $p$ contains a square in $\{s_{j+1}, s_{j+2}, \ldots, s_{l-1}\}$ (resp. $\{s_{i+1}, s_{i+2} \ldots, s_{k-1}\}$) if $\col(s_k)=\col(s_j)$ and $l-j \geq 3$ (resp. $k-i \geq 3$).

    \item $p$ contains a square in $\{s_1, s_2 \ldots, s_{i-1}\}$ if $\col(s_k) \neq \col(s_j)$ and $i >2$.

    \item  $p$ contains $s_{|\D|}$ if $\col(s_k) \neq \col(s_j)$ and $i=2$.
\end{enumerate}

Consider the plug $\hat{p}=p \cup d$.
By construction,
$$\tl_{d,p} \sim \tl_p^{-1} * (p, \emptyset, p^{-1})* (p^{-1}, \emptyset, p) * f * f_{\text{vert}} * (\hat{p}, \emptyset, \hat{p}^{-1}) * (\hat{p}^{-1}, \emptyset, \hat{p}) * \tl_{\hat{p}}$$ 
where $f=(p, d, \hat{p}^{-1})$ and $f_{\text{vert}}=(\hat{p}^{-1}, \emptyset, \hat{p})$.
Thus, the later tiling contains four vertical dominoes in $\tilde{d} \times [|p|+1, |p| +5]$.
Perform three flips (Figure~\ref{fig:ex3flips} shows these flips for a specific hamiltonian disk) to conclude that
$$\tl_{d,p} \sim \ldots * (p^{-1}, \tilde{d}, p\smallsetminus \tilde{d}) * (p\smallsetminus \tilde{d} , d, (\hat{p} \smallsetminus \tilde{d})^{-1}) * ((\hat{p} \smallsetminus \tilde{d})^{-1}, \emptyset, \hat{p} \smallsetminus \tilde{d}) * ((\hat{p} \smallsetminus \tilde{d}, \tilde{d}, (\hat{p} \smallsetminus \tilde{d})^{-1})* \ldots.$$
It now follows from Fact~\ref{fact: floors} that $\tl_{d,p} \sim \tl_{\tilde{d}, p \smallsetminus \tilde{d}}^{-1} * \tl_{d,p \smallsetminus \tilde{d}} * \tl_{\tilde{d}, \hat{p} \smallsetminus \tilde{d}}$.

\begin{figure}[H]
\centerline{
\includegraphics[width=0.97\textwidth]{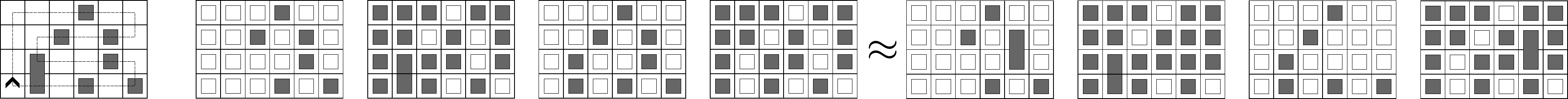}}
\caption{A domino $d$ and a plug $p$ in a hamiltonian disk, and the effect of thee flips in the tiling described by $(p^{-1}, \emptyset, p) * f * f_{\text{vert}} * (p \cup d, \emptyset, (p \cup d)^{-1})$.}
\label{fig:ex3flips}
\end{figure}

A careful analysis shows that under the four conditions above the desired inequality $\max \{|\flux_0(d,p \smallsetminus \tilde{d})|, |\flux_0(\tilde{d},\hat{p} \smallsetminus \tilde{d})|, |\flux_0(\tilde{d},p \smallsetminus \tilde{d})| \} < |\flux_0(d,p)|$ holds.
For the case $\col(s_j) \neq \sign(\flux_0(d,p))$, we derive similar conditions on $p$, which imply that $\tl_{d,p} \sim \tl_{\tilde{d},p} * \tl_{d,p\cup \tilde{d}} * \tl_{\tilde{d},p \cup d}^{-1}$ and $\max \{|\flux_0(d,p \cup \tilde{d})|, |\flux_0(\tilde{d}, p \cup d)|, |\flux_0(\tilde{d},p)| \} < |\flux_0(d,p)|$.
\end{proof}

Henceforth, we restrict our attention to tilings $\tl_{d,p}$ such that $|\flux_0(d,p)|=1$.
The objective is to demonstrate that for any two such tilings $\tl_{d,p}$ and $\tl_{\tilde{d},\tilde{p}}$, we have $\tl_{d,p} \sim \tl_{\tilde{d},\tilde{p}}^{ \pm 1}$.
We first show a partial result in this direction.

\begin{lemma}\label{lem:onegenerator}
Let $\D$ be a disk with a hamiltonian cycle $\gamma = (s_1, \ldots, s_{|\D|})$ where $s_1=s_{SW}$.
Consider a domino $\tilde{d} \in \D_{\gamma}$ and a plug $\tilde{p} \in \mathcal{P}_d$ such that $|\flux_0(\tilde{d},\tilde{p})|=1$.
Then, for each $d \in \D_{\gamma}$ there exists $p \in \mathcal{P}_d$ such that $|\flux_0(d,p)|=1$ and $\tl_{d,p}^{\pm 1} \sim \tl_{\tilde{d},\tilde{p}}$.
\end{lemma}

\begin{proof}
Take dominoes $d_1,d_2, \ldots, d_n$ as in Lemma~\ref{lem:alternatingdominoes}.
Then, $d_1=\tilde{d}$, $d_n = d$ and there are three possibilities for the relative position of $d_i$ and $d_{i+1}$.
We claim that there exist plugs $p_1, p_2, \ldots,p_n \in \mathcal{P}$ such that $p_1=\tilde{p}$, $|\flux_0(d_{i+1},p_{i+1})| = |\flux_0(d_i,p_i)|$ and $\tl_{d_{i+1},p_{i+1}} \sim \tl_{d_i,p_i}^{\pm 1}$.
Notice that the claim implies the desired result.

From now on, we use Lemma~\ref{lem:fluxinvariance} repeatedly without further mention.
In order to prove the claim, we analyze the possible relative positions of $d_i$ and $d_{i+1}$.
First, suppose that $d_{i+1}= s_k \cup s_l$ with $s_k \subset \D_{d_i,0}$ and $s_l \subset \D_{d_i,\pm 1}$.
Since $|\flux_0(d_i,p_i)| = 1$, we may assume that either $p_i=d_{i+1}$ or $p_i^{-1}=d_i \cup d_{i+1}$.
Now, as in the proof of Lemma~\ref{lem:generatorsfluxgeq2}, if $p_i=d_{i+1}$ then $\tl_{d_i, p_i} \sim \tl_{d_{i+1},\pl_{\circ}}^{-1} * \tl_{d_i,\pl_{\circ}} * \tl_{d_{i+1},d_i}$ and therefore, by Lemma~\ref{lem:tdvazio}, $\tl_{d_i,p_i} \sim \tl_{d_{i+1},d_i}$.
Similarly, if $p_i^{-1}=d_i \cup d_{i+1}$ then \mbox{$\tl_{d_i,p_i} \sim \tl_{d_{i+1}, \pl_{\bullet} \smallsetminus (d_i \cup d_{i+1})} * \tl_{d_i, \pl_{\bullet} \smallsetminus d_{i}} * \tl_{d_{i+1} , \pl_{\bullet} \smallsetminus d_{i+1}}^{-1}$} and consequently $\tl_{d_i,p_i} \sim \tl_{d_{i+1}, \pl_{\bullet}\smallsetminus (d_i\cup d_{i+1})}$.
In both cases we obtain a plug $p_{i+1}$ such that $\tl_{d_i,p_i} \sim \tl_{d_{i+1},p_{i+1}}$ and $|\flux_0(d_{i+1},p_{i+1})| = |\flux_0(d_i,p_i)|$.

Secondly, suppose that $d_{i+1}=s_k \cup s_l$, where $s_k,s_l \subset \D_{d_i,0}$, and that $\D_{d_i,0} \smallsetminus \D_{d_{i+1},0}$ is a disk.
Then, $\D_i= (\D_{d_i,0}\cup d_i) \smallsetminus \D_{d_{i+1},0}$ is also a disk and has a hamiltonian cycle $\gamma_i$ induced by $\gamma$.
Let $\gamma_{i,1}$ and $\gamma_{i,2}$ be the paths obtained from $\gamma_i$ by disconnecting the adjacent unit squares in $d_i$ and $d_{i+1}$.
There are two cases: $|\gamma_{i,1}|$ and $|\gamma_{i,2}|$ are both either even or odd.

We obtain tilings $\tl_{i,1} \in \mathcal{T}(\D_i \times [0,1])$ and $\tl_{i,2} \in \mathcal{T}((\D_i \smallsetminus d_i) \times [0,1])$ by placing dominoes along $\gamma_{i,1}$ and $\gamma_{i,2}$; for instance, see Figure~\ref{fig:t1t2}.
If $|\gamma_{i,1}|$ and $|\gamma_{i,2}|$ are even then $(d_i \cup d_{i+1}) \times [0,1] \not\in \tl_{i,1}$ and $d_{i+1} \times [0,1] \in \tl_{i,2}$.
If $|\gamma_{i,1}|$ and $|\gamma_{i,2}|$ are odd then $d_i \times [0,1] \not\in \tl_{i,1}$, $d_{i+1} \times [0,1] \in \tl_{i,1}$ and $d_{i+1} \times [0,1] \not\in \tl_{i,2}$.
Let $f_{i,1}^*$ and $f_{i,2}^*$ be the set of planar dominoes that describes $\tl_{i,1}$ and $\tl_{i,2}$, respectively.

\begin{figure}[H]
\centerline{
\includegraphics[width=0.378\textwidth]{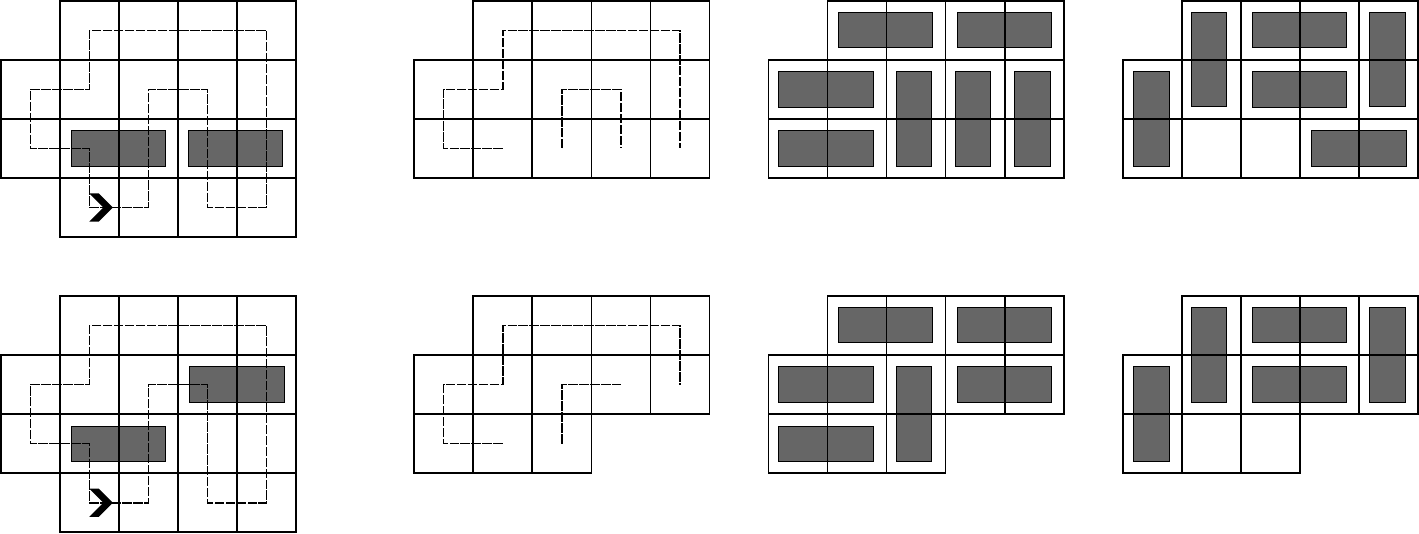}}
\caption{Two examples of a disk with a hamiltonian cycle $\gamma$ and dominoes $d_i$ and $d_{i+1}$, the disk $\D_i$ with paths $\gamma_{i,1}$ and $\gamma_{i,2}$, and the tilings $\tl_{i,1}$ and $\tl_{i,2}$.
In the first (resp.\ second) example $|\gamma_{i,1}|$ and $|\gamma_{i,2}|$ are both even (resp.\ odd).}
\label{fig:t1t2}
\end{figure}

Consider the tiling $\tl_{d_i,p_i} = \tl_{p_i}^{-1} * f * f_{\text{vert}} * \tl_{p_i \cup d_i}$.
We may assume that $p_i \cap \D_{d_i,0} \subset \D_{d_{i+1},0}$, as $|\flux_0(d_i,p_i)|=1$.
Then, $\tl_{d_i,p_i}$ covers the region $(\D_{d_i,0} \smallsetminus \D_{d_{i+1},0}) \times [|p_i|,|p_i|+2]$ only with vertical dominoes.
Since $\D_i$ is a disk, as in the proof of Lemma~\ref{lem:tdvazio}, we have a sequence of flips that takes $\tl_{d_i,p_i}$ to the tiling $\tl=\tl_{p_i}^{-1} * (p_i, f_{i,1}^*, (\D_i \cup p_i)^{-1}) * ((\D_i \cup p_i)^{-1}, f_{i,2}^*, p_i \cup d_i) * \tl_{p_i \cup d_i}$.
From Facts~\ref{fact: floors} and~\ref{fact:gen}, if $|\gamma_{i,1}|$ and $|\gamma_{i,2}|$ are odd (resp.\ even) then $\tl$ is $\sim$-equivalent to $\tl_{d_{i+1},p_i}$ (resp.\ $\tl_{d_{i+1}, p_i \cup d_i}^{-1}$). 
Thus, in any case, we have a plug $p_{i+1}$ such that $\tl_{d_i,p_i} \sim \tl_{d_{i+1},p_{i+1}}^{\pm 1}$ and $\flux_0(d_{i+1},p_{i+1}) = \flux_0(d_i,p_i)$.
Finally, notice that a completely analogous argument holds in the third possible case, when $d_{i+1}=s_k \cup s_l$ with $s_k,s_l \subset \D_{d_i, \pm 1}$, and $\D_{d_i, \pm 1} \smallsetminus \D_{d_{i+1}, \pm 1}$ is a disk.
\end{proof}

For hamiltonian disks without bottlenecks, we now prove the existence of a domino that, in some sense, connects the family of tilings with flux equals $+1$ and the family of tilings with flux equals $-1$.
As a corollary, we obtain that the even domino group of a bottleneck-free hamiltonian disk is cyclic.

\begin{lemma}\label{lem:inversetilingplug}
Let $\D$ be a nontrivial disk with a hamiltonian cycle $\gamma=(s_1,  \ldots, s_{|\D|})$ where $s_1=s_{SW}$.
If there is no domino that disconnects $\D$ then there exists a domino $d\in \D_{\gamma}$ such that $\tl_{d,p}^{-1} \sim \tl_{d,p^{-1} \smallsetminus d}$ for every plug $p \in \mathcal{P}_d$ with $\flux_0(d,p)=1$.
\end{lemma}

\begin{proof}
Notice that $s_1$ is contained in a $3 \times 3$ square $\widetilde{\D}$, as no domino disconnects $\D$ and $s_1=s_{SW}$.
Suppose that in $\widetilde{\D}$ the cycle $\gamma$ follows one of the three patterns illustrated in Figure~\ref{fig:pathtypes}; the other cases are obtained by reversing the orientation, and the argument follows similarly.
We say that $\gamma$ is of type $i$ if it follows the $i$-th possible pattern.

\begin{figure}[H]
\centerline{
\includegraphics[width=0.25\textwidth]{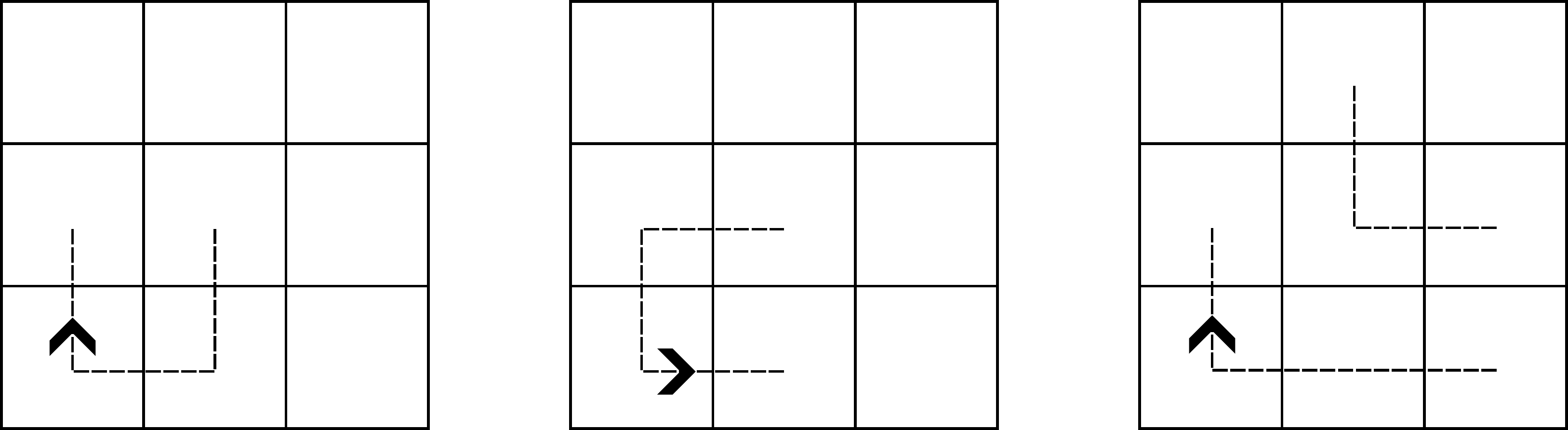}}
\caption{Three possible patterns of $\gamma$.}
\label{fig:pathtypes}
\end{figure}

Let $d$ be the domino adjacent and parallel to $s_1 \cup s_{|\D|}$.
Suppose that $\gamma$ is either of type 1 or type 3; the case where $\gamma$ is of type 2 is analogous to the case where $\gamma$ is of type 1.
Consider a plug $p \in \Pl_d$ such that $\flux_0(d,p)=1$; consequently, $\flux_0(d,p^{-1} \smallsetminus d)=-1$.
By Lemma~\ref{lem:fluxinvariance}, it suffices to consider the case in which $p^{-1} \smallsetminus d= s_{|\D|} \cup s_i$, where $s_i \subset \widetilde{\D} \smallsetminus d$ denotes the unit square adjacent to $s_2$.

Consider the tiling $\tl_{d,p^{-1} \smallsetminus d} = \tl_{p^{-1}\smallsetminus d}^{-1} * f * f_{\text{vert}} * \tl_{p^{-1}}$.
Insert vertical floors around $f$ and $f_{\text{vert}}$ to obtain a tiling $\tl_1 \sim \tl_{d,p^{-1} \smallsetminus d}$; the restriction of $\tl_1$ to $\widetilde{\D} \times [|p^{-1} \smallsetminus d|, |p^{-1} \smallsetminus d|+7]$ is shown in the first row of Figure~\ref{fig:tdp-1}.
Let $\tl_i$ be the tiling whose restriction to the region $\widetilde{\D} \times [|p^{-1} \smallsetminus d|, |p^{-1} \smallsetminus d|+7]$ is as in the $i$-th first row of Figure~\ref{fig:tdp-1}, and which is equal to $\tl_1$ outside this region.
We have $\tl_1 \approx \tl_2$ and $\tl_3 \approx \tl_4 \approx \tl_5$.
Notice that the restrictions of $\tl_2$ and $\tl_3$ to $\widetilde{\D} \times [|p^{-1} \smallsetminus d|+2, |p^{-1} \smallsetminus d|+6]$ define two tilings, both with the same twist, of a $3 \times 3  \times 4$ box.
Since the $3 \times 4$ rectangle is regular (see Lemma 9.2~of~\cite{Sal22}), it follows that $\tl_2 \sim \tl_3$.

\begin{figure}[H]
\centerline{
\includegraphics[width=0.42\textwidth]{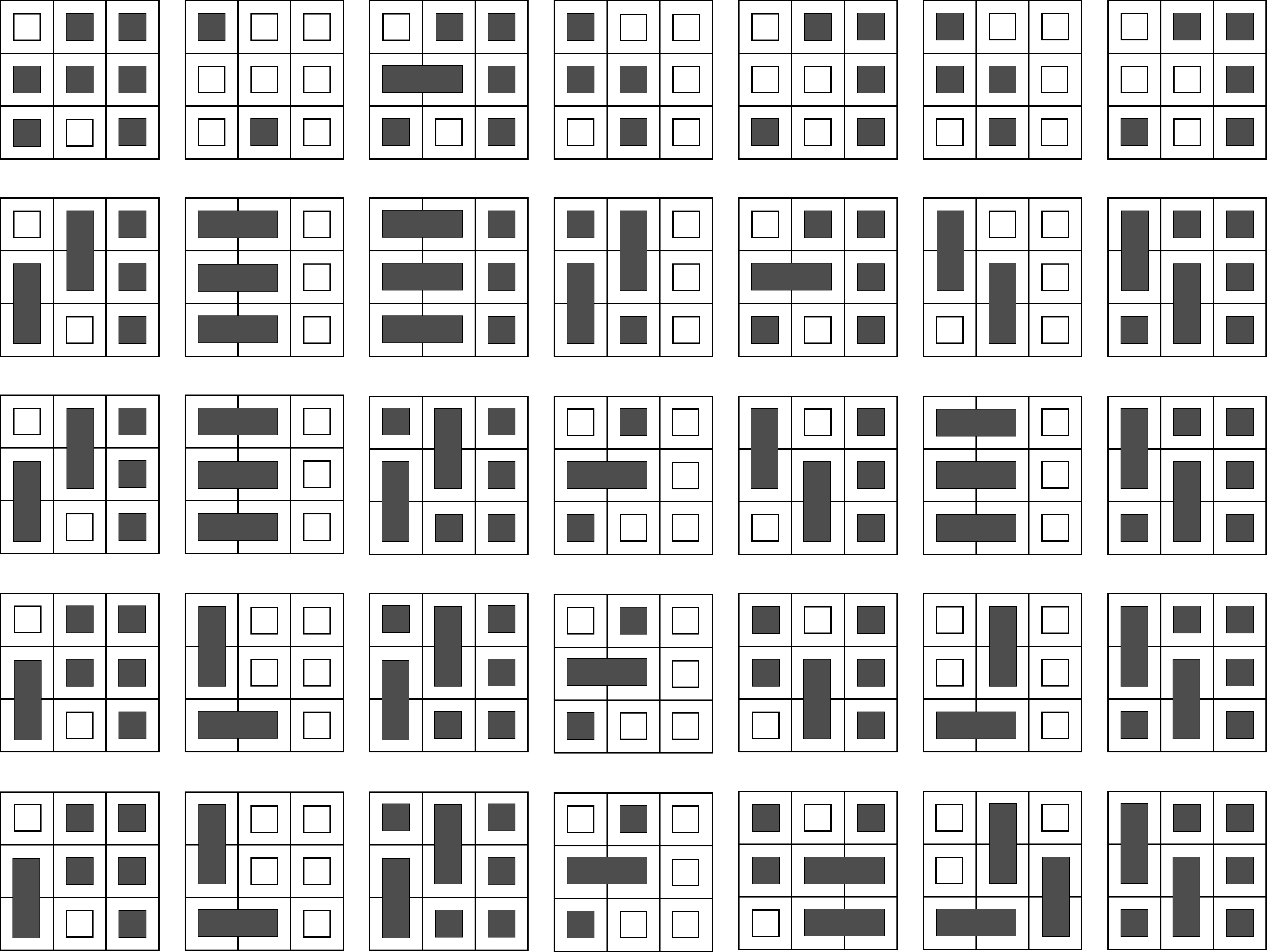}}
\caption{The $i$-th row shows the restriction to $\widetilde{\D}$ of seven floors of a tiling $\tl_i$.}
\label{fig:tdp-1}
\end{figure}

We focus on $\tl_4$ when $\gamma$ is of type 1.
Let $f_1, f_2, \ldots, f_7$ be the seven floors of $\tl_4$ whose restriction to $\widetilde{\D}$ is shown in Figure~\ref{fig:tdp-1}.
It follows from Facts~\ref{fact:resptil} and~\ref{fact: floors} that $\tl_4 \sim \tl_{f_2^{-1}}^{-1} * \tl_{f_3} * \tl_{f_4^{-1}}^{-1} * \tl_{f_6^{-1}}^{-1} * \tl_{f_7}$, since all other floors of $\tl_4$ contain only dominoes that respect the cycle $\gamma$.
By Fact~\ref{fact:gen} and Lemma~\ref{lem:tdvazio}, we then have $\tl_4 \sim \tl_{f_4^{-1}}^{-1}$.
On the other hand, $\tl_{f_4^{-1}}^{-1} = \tl_{d, s_1 \cup s_j}^{-1}$, where $s_j \subset \widetilde{\D} \smallsetminus s_2$ is the unit square adjacent to $s_i$.
Finally, by Lemma~\ref{lem:fluxinvariance}, $\tl_{d, s_1 \cup s_j}^{-1} \sim \tl_{d,p}^{-1}$.
Similarly, $\tl_5 \sim \tl_{d,p}^{-1}$ when $\gamma$ is of type 3.
\end{proof}

\begin{corollary}\label{col:cyclic}
Let $\D$ be a nontrivial hamiltonian balanced quadriculated disk.
Suppose that there is no domino that disconnects $\D$.
Consider dominoes $d_1,d_2 \in \D_{\gamma}$ and plugs $p_1,p_2 \in \mathcal{P}$.
If $|\flux_0(d_1,p_1)|=1=|\flux_0(d_2,p_2)|$ then $\tl_{d_1,p_1}^{\pm 1}\sim \tl_{d_2,p_2}$.
\end{corollary}

\begin{proof}
Take the domino $d$ as in Lemma~\ref{lem:inversetilingplug}.
By Lemma~\ref{lem:onegenerator}, there exist plugs $p$ and $\tilde{p}$ with $|\flux_0(d,p)|=1=|\flux_0(d,\tilde{p})|$ such that $\tl_{d_1,p_1} ^{\pm 1} \sim \tl_{d,p}$ and $\tl_{d_2,p_2} ^{\pm 1} \sim \tl_{d,\tilde{p}}$.
The result now follows from Lemma~\ref{lem:fluxinvariance}.
Indeed, we have either $\tl_{d,p} \sim \tl_{d,\tilde{p}}$ or $\tl_{d,p} \sim \tl_{d,p^{-1}\smallsetminus d} ^{-1}\sim \tl_{d,\tilde{p}} ^{-1}$, depending on whether
$\flux_0(d,p) = \flux_0(d,\tilde{p})$ or $\flux_0(d,p)= -\flux_0(d,\tilde{p})$.
\end{proof}

The proof of Theorem~\ref{thm:regdisks} follows from combining the established results. 

\begin{proof}[Proof of Theorem~\ref{thm:regdisks}]
Since $\D$ is nontrivial, the twist $\twist \colon G_{\D}^+ \to \mathbb{Z}$ is a surjective homomorphism (see Lemma 6.2 of~\cite{Sal22}).
It follows from Lemma~\ref{lem:generatorsfluxgeq2} and Corollary~\ref{col:cyclic} that the even domino group $G_{\D}^+$ is cyclic.
Therefore, the twist is an isomorphism.    
\end{proof}

Our approach to prove Theorem~\ref{thm:bottleneck} relies on utilizing the established regularity of bottleneck-free hamiltonian disks.
To this end, we need an additional lemma concerning generators $\tl_{d,p}$ where $\D \smallsetminus d$ is not connected.

\begin{lemma}\label{lem:gendisc}
Consider a disk $\D$ satisfying the hypothesis of Theorem~\ref{thm:bottleneck}.
Let $d \in \D_{\gamma}$ be a domino that disconnects $\D$ and consider a plug $p \in \mathcal{P}_d$.
Then, $\tl_{d,p}$ is $\sim$-equivalent to a concatenation of tilings of the form $\tl_{\tilde{d},\tilde{p}} ^{\pm 1}$ with $\D \smallsetminus \tilde{d}$ connected.
\end{lemma}

\begin{proof}
We first distinguish four planar dominoes.
Suppose that $d \subset \D_i$ for some $i>0$, otherwise set $d_1=d$.
Let $d_1 \subset \D_0$ be the domino that contains the line segment of length two defined by $\D_0$ and $\D_i$.
Thus, $\D \smallsetminus d_1$ is not connected.
Let $d_2 \subset \D_0$ be the domino adjacent and parallel to $d_1$ and denote by $d_3$ and $d_4$ the dominoes obtained after performing a flip on $d_1$ and $d_2$.
Notice that $\D \smallsetminus d_2$ is connected, and at least one of the regions $\D \smallsetminus d_3$ or $\D \smallsetminus d_4$ is also connected.

Assume, without loss of generality, that the initial unit square of the hamiltonian cycle of $\D$ is contained in $\D_0$.
We then have $\D_{d_1,0} = \D_i$.
Notice that $\D_{d_1,0} \smallsetminus \D_{d,0}$ is a hamiltonian disk, as $d$ and $d_1$ disconnect $\D$.
Since $|\D_i| < |\D_0| -2$, we can modify $p$ to obtain a plug $p_1$ such that $p_1 \subset \D_{d,0} \cup \D_0 \smallsetminus d_1$, $d_2 \subset p_1^{-1}$ and $\flux_0(d,p)=\flux_0(d,p_1)$.
For an example of this construction, see Figure~\ref{fig:genbottleneck}.
It follows from Lemma~\ref{lem:fluxinvariance} that $\tl_{d,p} \sim \tl_{d,p_1}$.
Furthermore, as in the proof of Lemma~\ref{lem:tdvazio}, we have $\tl_{d,p_1} \sim \tl_{d_1,p_1}^{\pm 1}$.

\begin{figure}[H]
\centerline{
\includegraphics[width=0.55\textwidth]{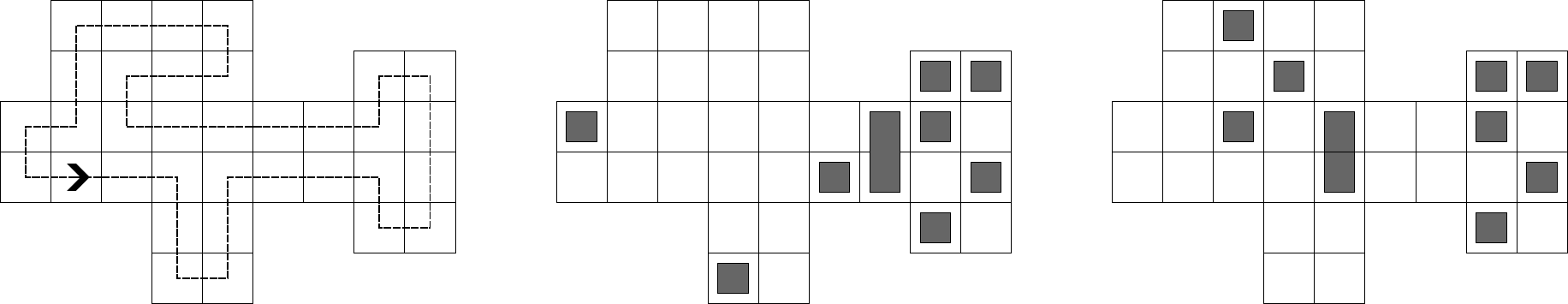}}
\caption{A hamiltonian disk $\D$ where $\D_0$ is equal to a $4\times 4$ square, a domino $d$ with a compatible plug $p$, and the domino $d_1$ with a plug $p_1$.}
\label{fig:genbottleneck}
\end{figure}

Since $d_2 \subset p_1^{-1}$, a vertical flip modifies $\tl_{d_1,p_1}$ to now contain the dominoes $d_2 \times [K,K+1]$ and $d_2 \times [K+1,K+2]$, where $K=|p_1|$.
A horizontal flip then takes $d_1 \times [K,K+1]$ and $d_2 \times [K,K+1]$ to $d_3 \times [K,K+1]$ and $d_4 \times [K,K+1]$.
Thus, Facts~\ref{fact: floors} and~\ref{fact:gen} imply that $\tl_{d_1,p} \sim \tl_{d_3, p} * \tl_{d_4, p\cup d_3} * \tl_{d_2, p\cup d_1}^{-1}$.

Therefore, the tiling $\tl_{d,p}$ is $\sim$-equivalent to a concatenation of tilings of the form $\tl_{\tilde{d}, \tilde{p}}^{\pm 1}$.
By Fact~\ref{fact:resptil}, we have $\tl_{\tilde{d},\tilde{p}}^{\pm 1} \sim \tl_{\text{vert}}$ if $\tilde{d} \not\in \D_{\gamma}$.
Otherwise, by construction, $\tilde{d}$ is either a domino that does not disconnect $\D$ or satisfies $\flux_0(\tilde{d},\tilde{p})=0$.
The result then follows from Lemma~\ref{lem:tdvazio}.
\end{proof}

\begin{proof}[Proof of Theorem~\ref{thm:bottleneck}]
Lemmas~\ref{lem:generatorsfluxgeq2},~\ref{lem:onegenerator} and~\ref{lem:gendisc} imply that the even domino group $G_\D^+$ is generated by tilings $\tl_{d,p}$ with $d \subset \D_0$ and $p \in \mathcal{P}_d$ such that $\D \smallsetminus d$ is connected and $|\flux_0(d,p)|=1$.
By Lemma~\ref{lem:fluxinvariance}, we may further restrict to plugs contained in $\D_0$.
We claim that each such tiling $\tl_{d,p}$ is equivalent under $\sim$ to a tiling of $\D \times [0,2N]$ whose restriction to $(\D \smallsetminus \D_0) \times [0,2N]$ is composed solely of vertical dominoes.
Notice that the desired result follows from the claim, since Theorem~\ref{thm:regdisks} implies that $\D_0$ is regular.

The hamiltonian cycle $\gamma$ of $\D$ induces a hamiltonian cycle $\gamma_0$ of $\D_0$.
The dominoes that respect $\gamma_0$ and does not respect $\gamma$ are exactly the dominoes $d_i \subset \D_0$ that contain the line segment defined by $\D_0$ and $\D_i$.
Let $\tl$ be the tiling of $\D \times [0,2N]$ obtained from the tiling $\tl_{d,p;\gamma_0}$ of $\D_0 \times [0,2N]$ by covering the region $(\D \smallsetminus \D_0) \times [0,2N]$ with vertical dominoes.
It follows from Facts~\ref{fact: floors} and~\ref{fact:gen} that $\tl$ is $\sim$-equivalent to a concatenation of $\tl_{d,p;\gamma}$ and possibly tilings $\tl_{d_i,p_i;\gamma}^{\pm 1}$ such that $\flux_0(d_i,p_i)=0$.
Thus, by Lemma~\ref{lem:tdvazio}, we have $\tl \sim \tl_{d,p;\gamma}$ and the claim follows.
\end{proof}

\bibliographystyle{plain}
\bibliography{bibliography}
\end{document}